\documentclass{kms-j}
\newcommand{\publname}{----------}
\issueinfo{}
  {}
  {}
  {}
\pagespan{1}{}
\copyrightinfo{}
  {The Korean Mathematical Society}
\usepackage{amsmath, amsthm, amscd, amsfonts,amssymb, graphicx, color} 
\usepackage[ansinew]{inputenc}
\usepackage{graphicx}

\newtheorem{thm}{Theorem}[section] 
\newtheorem{lem}[thm]{Lemma}
\newtheorem{prop}[thm]{Proposition}
\newtheorem{cor}[thm]{Corollary}
\theoremstyle{definition}
\newtheorem{defn}[thm]{Definition}

\theoremstyle{remark}

\newcommand{\h}{\mathcal{H}}

\newcommand{\s}{\quad }
\theoremstyle{plain}

\theoremstyle{definition}

\theoremstyle{remark}

\begin{document}
\title[ Pair Frames]
{ Pair Frames }

\author{Abolhassan Fereydooni}
\address{Department of Basic Sciences \\ Ilam University \\Ilam, Iran}
\email{fereydooniman@yahoo.com,  a.fereydooni@mail.ilam.ac.ir }

\author{Ahmad Safapour}
\address{Department of Mathematics\\Vali-e-Asr University of Rafsanjan\\ Rafsanjan, Iran}
\email{safapour@vru.ac.ir}
\begin{abstract}
In this paper, a new concept related to the frame theory is introduced; the notion of \textit{pair frame}. By investigating some properties of such frames, it is shown that pair frames are a generalization of ordinary frames. Some classes of pair frames are considered such as \textit{$(p,q)$-pair frames} and \textit{near identity pair frames}.
\end{abstract}
\maketitle


\noindent \vskip 0.8 true cm 
\section{\bf {\bf \em{\bf Introduction}}} 
%

In 1946, Gabor \cite{1946TheoryofGabor} introduced a method for reconstructing functions (signals) using a family of elementary functions. Later in 1952, Duffin and Schaeffer \cite{1951NonHorDuff} presented a similar tool in the context of nonharmonic Fourier series and this is the starting point of frame theory. After some decades, Daubechies, Grossmann and Meyer \cite{1986PainlessDube} announced formally the definition of frame in the abstract Hilbert spaces.

In the past two decades, frame theory has become an interesting and fruitful field of mathematics with abundant applications in other branches of sciences. 
There are numerous applications of operators in the definition of pair fram, multiplier operators. As a particular way to implement time-varian fitters, Gabor frame multipliers are used \cite{2003 A First Survey of Gabor Multipliers - Feichtinger Nowak}; also known as Gabor filters \cite{2002 Linear Time-Frequency Filters- Matz  Hlawatsch}. Such operators find applications in psychoacoustics \cite{2010 Time-Frequency Sparsity by Removing Perceptually - Balazs  Laback}, denoising \cite{2011 A Time-Frequency Method for Increasing the- Majdak  Balazs}, computational auditory scene analysis \cite{2006 Computational Auditory Scene Analysis- Wang  Brown}, virtual acoustics \cite{2007 Multiple Exponential Sweep Method for - Majdak  Balazs} and seismic data analysis \cite{2005 The Gabor Transform Pseudodifferential- Margrave  Gibson}. 

The main idea of the frame theory is reconstructing elements of a function space using some special subsets of it.
The interested readers are referred to \cite{2008FrameBookChris} and \cite{2000ArtofCasa} for more details. Some generalizations of frame significance have been presented such as fusion frames(frame of subspaces) \cite{2003FramSubsCasa}, generalized frames \cite{2006G-framSun} and continuous frames \cite{2005ContinuousFramesFornasie}. \\

The present paper is organized as follows: In section 2, some conditions equivalent to the concept of frames are introduced. These equivalent condition are motivations of defining \emph{pair frames} and \emph{near identity pair frames}. Some other equivalent conditions of frames, based on the pair frame operator, are derived there. The importance and role of pair frames in frame theory in three useful aspect are considered. Also some frame-like inequalities for pair frames are given.
The (\emph{$p,q$)-pair frames (Bessels)} are introduced in section 3. The notion of \emph{near identity pair frame} is presented in section 4 and some results related to this are proved.

Here, we recall the definition of frames and some preliminary notations. The notation $\mathbb{I}$ denotes a countable index set and the subscripts $i$ belong to $\mathbb{I}$. $ \mathcal{H} $ is a Hilbert space with inner product $ \langle .,. \rangle $. The bounded operators on $\h$ denoted by $ \mathcal{B } (\h)$. A frame in $\h$ is defined as below.\\ 


\begin{defn} \label{d: frame} 
A subset $ F= \{ f_i \} \subset \mathcal{H} $ is a \textbf{frame} for $ \mathcal{H}$ if there exist constants $A,B>0$ such that for every $f \in \h $ 

\begin{equation}\label{e: ordinary frame} 
A\|f\|^2\leqslant\sum_{i} | \langle f , f_i \rangle |^2\leqslant B\|f\|^2 .
\end{equation} 
Constants $A $ and $B$ are called lower and upper frame bounds, respectively. If the upper inequality in (\ref{e: ordinary frame}) holds, $ F$ is said to be a \textbf{Bessel sequence} with bound $B$.
\end{defn} 

If there is another Bessel sequence $ G=\{ g_i \}\subset \h$ satisfying \begin{equation}\label{e: expanition} f=\sum_i \langle f, g_i \rangle f_i , \end{equation} for every $f \in \h $, $ G $ is said to be a \textbf{dual} of $F $. The above identity shows that for reconstructing $f \in \mathcal{H}$ we need a sequence of scalars $\{ \langle f, f_i \rangle \}$ and another sequence $ G=\{ g_i \}$ of vectors of $\h$. Sometimes in this paper, $F $ in the Definition \ref{d: frame} will be called an ordinary frame for $\h$ in contrast of the other types of frames which will be introduced. 
\begin{prop}\label{p: Bessels-pair Bessels} 
A family $F=\{f_i\}\subset \mathcal{H}$ is a Bessel sequence for $ \mathcal{H}$ if and only if the operator 
\begin{equation}\label{e:Bessel Operator} 
S (S_F): \mathcal{H} \rightarrow \mathcal{H}, \s S(f):= \sum \langle f, f_i \rangle f_i,
\end{equation} 
is well-defined operator. In this situation, $S$ is a bounded operator.
\end{prop} 
\begin{proof} 
Suppose that $F=\{f_i\}$ is a Bessel sequence. Then there is a constant $B>0$ such that 
$$ \sum | \langle f,f_i \rangle |^2 \leqslant B \| f \|^2 , $$ 
for every $f\in \mathcal{H}$. Therefore 
\[ U : \mathcal{H} \rightarrow \ell^2 , \quad U(f)=\{ \langle f, f_i \rangle \}, \] 
is a well-defined and bounded operator. Also its adjoint is 
\[U^* : \ell^2 \rightarrow \mathcal{H}, \quad U^*(\{c_i\})= \sum c_i f_i . \] 
Hence the operator 
\[ S:= U^*U: \mathcal{H} \rightarrow \mathcal{H}, \quad S(f)=\sum \langle f, f_i \rangle f_i , \] 
is a well-defined and bounded.
For the converse, let the operator defined in (\ref{e:Bessel Operator}) be a well-defined and bounded. Then for every $f \in \mathcal{H}$,
\[\sum |\langle f, f_i \rangle |^2= \langle S(f) , f\rangle = | \langle S(f) , f\rangle | \leqslant \|S\| \|f\|^2 . \] 
\end{proof} 

Let $G= \{ g_i \}, F=\{ f_i \} \subset \h $ be Bessel sequences for $ \h$. Define 
\begin{align*} 
U_G : \h \to \ell^2 & , \s U_G(f)= \{ \langle f, g_i \rangle\},\\ 
T_F : \ell^2 \to \h &, \s T_F( \{c_i\})= \sum c_i f_i .
\end{align*} 

For $V \in \mathcal{B}(\mathcal{H})$, $ \mathcal{N}(V) $ and $ \mathcal{R}(V)$ denote the kernel and the range of $T$, respectively.
A partial order on $ \mathcal{B}(\mathcal{H})$ can be defined as follows. For $V, W \in \mathcal{B}(\mathcal{H})$ we write $V \leqslant W $ $ (V < W)$ whenever for every $f \in \h $ 
\[ \langle V(f),f \rangle \leqslant \langle V (f) , f \rangle \s ( \langle V (f) , f \rangle < \langle V(f) , f \rangle , f\neq 0 ) . \] 
A self adjoint operator $V \in \mathcal{B}(\mathcal{H}) $ is called nonnegative (positive) if for every $f \in \h $ 
\[ 0 \leqslant \langle V(f) , f \rangle \quad \s (0 < \langle V (f) , f \rangle , f\neq 0) . \] 
\noindent \vskip 0.8 true cm 

\section{\bf {\bf \em{\bf Pair Frames}}} 
%


Next theorem presents some equivalent conditions for frames and
is a useful motivation for defining pair frames and near
identity pair frames.


\begin{thm}\label{t: frame}
For a sequence $F=\{ f_i\} \subset \h $ and its corresponding
operator $S_F $ defined in (\ref{e:Bessel Operator}), the
following statements are equivalent:
\begin{enumerate}
\item $F $ is a frame for $\h$.
\item $S_F$ is well defined, bounded and there exist constants
$A,B>0$ such that
\[ A\leqslant S_F \leqslant B.\]
\item $S_F$ is well defined, bounded and there exists
$\alpha\in(0,\infty)$ such that
\[ \|I-\alpha S_F\| < 1. \]
\item $S_F$ is well defined, bounded and invertible.
\item $S_F$ is well defined, bounded and surjective.
\end{enumerate}
\end{thm}
\begin{proof}
From Proposition \ref{p: Bessels-pair Bessels} we know that the
sequence $F $ is a Bessel sequence for $\h$ if and only if
$S_F$ is a well defined and bounded operator. Using notations
mentioned before and the proof of Proposition \ref{p:
Bessels-pair Bessels}, it follows that there is a constant
$B>0$ such that for every $f \in \h $,

$$ 0 \leqslant \sum_i |\langle f ,f_i \rangle |^2 \leqslant B \|f\|^2 .$$

Also since for every $f \in \h $,
\begin{equation}\label{e:Bessel equality}
\sum_i |\langle f ,f_i \rangle |^2= \langle S_F (f) , f\rangle ,
\end{equation}
$$ $$
then
\[ 0 \leqslant S_F \leqslant B.\]
$(1)\Leftrightarrow(2).$ $ F $ is a Bessel sequence if and only if $S_F$ is well defined and $0 \leqslant S_F \leqslant B$. The inequality notion of operators on Hilbert spaces and equation (\ref{e:Bessel equality}) prove the equivalence $(1)\Leftrightarrow(2).$ \\
$(2) \Rightarrow (3). $ It is proved in the standard textbooks of frame theory (see for example \cite{2008FrameBookChris}). But we restate the proof here.
Since $A \leqslant S_{F} \leqslant B$, then
$$ \frac{A}{B} \leqslant \frac{1}{B}S_{F} \leqslant 1 .$$
Therefore,
$$ 0 \leqslant 1- \frac{1}{B}S_{F} \leqslant 1- \frac{A}{B} \lneqq 1.$$
By putting $\alpha= \frac{1}{B}$, we get
$$ \|I- \alpha S_{F} \| \leqslant 1- \frac{A}{B} \lneqq 1 . $$
$ (3)\Rightarrow (2).$ Let $ \|I-\alpha S_F\| < 1 $ for
some $\alpha \in (0, \infty)$. Put $ C=1 /\alpha $. Then a
positive number $D$ can be found such that $ 0< D<C$. So
\begin{equation}\label{e:5}
\|I- \frac{1}{C} S_F\| < D/C < 1.
\end{equation}
By putting $ A=C-D$, we have
\[ 0\leqslant \| I- \frac{{1}}{C}S_F \| \leqslant \frac{C-A}{C}< 1.\]
On the other hand, for every $f \in \h $ with $\| f \|=1$,
\begin{equation}\label{e: 8}
\frac{C-A}{C} = \frac{C-A}{C} \|f\|^2= \langle \frac{C-A}{C}f , f \rangle .
\end{equation}
Since $ S_F $ and hence $ I- \frac{{1}}{C}S_F$ is self-adjoint, by using \ref{e:5}
\begin{equation}\label{e: 10}
\sup _{\| f \| = 1} | \langle (I- \frac{{1}}{C} S_F ) (f) , f \rangle | = \| I- \frac{{1}}{C}S_F \| \leqslant \frac{D}{C} = \frac{C-A}{C}.
\end{equation}
Self-adjointness of $ I- \frac{{1}}{C} S_F $ implies that $ \langle
(I- \frac{{1}}{C} S_F ) f , f \rangle $ to be a real number for every
$f \in \mathcal{H}. $ By keeping this in mind and using
(\ref{e: 8}) and (\ref{e: 10}) we get
\[ \langle (I- \frac{{1}}{C} S_F ) (f) , f \rangle \leqslant \langle\frac{C-A}{C} f , f \rangle .\]
for every $f \in \h $ such that $\| f \|=1$. A simple calculation implies that for every $f \in \h $ with $\| f \|=1$
\[ \langle Af ,f \rangle \leqslant \langle S_F (f) , f \rangle .\]
Therefore, for all $f \in \h $,
\begin{equation}\label{e: 12}
A\| f\| ^2=\langle Af ,f \rangle \leqslant \langle S_F (f) , f \rangle .
\end{equation}
Since for every $f \in \h $
\begin{equation}\label{e: 13}
\langle Af ,f \rangle = A \| f \| ^2 \geq 0 ,
\end{equation}
the relation (\ref{e: 12}) yields that $ \langle S_F (f) , f
\rangle $ is nonnegative for every $ f \in \mathcal{H} $. So
$$ \langle S_F (f) , f \rangle = |\langle S_F (f) , f \rangle | ,$$
for every $f \in \h $. By putting $ B=\| S_F \| $ we obtain
\begin{equation}\label{e: 14}
\begin{split}
\langle S_F (f) , f \rangle =| \langle S_F (f) , f \rangle | \leqslant \| S_F \| \| f \| ^2 = B \| f \| ^2 = \langle B f,f \rangle .
\end{split}
\end{equation}
for every $f \in \h $. Relations (\ref{e: 12}) together with (\ref{e:
14}) implies that
\[ A \leqslant S_F \leqslant B. \]
$(3) \Rightarrow(4)$. It is a well-known result in the operator theory (see \cite{1980FuncConw}). \\
$(4) \Rightarrow(2)$. It is proved in Proposition 2.7 of \cite{2008G-framNaja-Farou}, but for the
sake of completeness we restate the proof. The operator $S_F$ is bounded and
positive, therefore $\sigma(S_F)\subset [0,\infty)$. Invertibility of $S_F$
implies that $ \sigma(S_F)$ does
not contain zero. We know that $\sigma(S_F)$ is compact. Hence there are nonnegative numbers $A$ and $B$
with $ A\leqslant B < \infty $ such that $\sigma(S_F) \subset [A,B]$.
Since $0 \notin \sigma(S_F)$ we can choose $ A>0$. Therefore $A \leqslant S_F \leqslant B$. \\
$ (4) \Rightarrow (5). $ It is obvious.\\
$ (5) \Rightarrow (4). $ It is enough to show that $S_{F}$ is one to one. Since $S_{F}$ is self adjoint and onto then
\[ \mathcal{N}(S_F)= \mathcal{N}(S_{F}^*) = {\mathcal{R}(S_F)}^\perp =\{ 0 \}. \]
\end{proof}
Well defindness and invertibility of $S_F $ in Theorem
\ref{t: frame} is the essence definition of pair Bessels and pair frames.

\begin{defn}\label{d: pair Bessel} 
For two sequences $F=\{f_i\}$, $G=\{ g_i\} \subset \h $ and a scalar sequence $m=\{m_i\}$, we say the triple $(m, G,F )$ 
is an \textbf{$m$-pair Bessel} for $\h$ if the operator 
\begin{equation}\label{e: pair Bessel} 
S_{mFG} : \mathcal{H} \rightarrow \mathcal{H}, \s S_{mFG} (f) =\sum_{i}m_i\langle f, g_i \rangle f_i,
\end{equation} 
is well-defined. i.e. the series converges for every $f \in \mathcal{H} $.
If the series converges unconditionally we will call $(m, G,F )$ an \textbf{unconditional $m$-pair Bessel}. If $m=\{1\}$, $ ( G,F )$ will be called a \textbf{pair Bessel}.
\end{defn} 
By principle of uniform boundedness, the operator $ S_{mFG}$ is bounded.
For $V \in \mathcal{B}(\mathcal{H})$ and $F=\{f_i\} \subset \h$, define
$$V F =\{V f_i\}, \s {m}F=\{ {m_i} f_i \}. \s $$ 
Clearly $S_{mFG}=T_{mF} U_G$. This describes the reason of using the notation $ S_{mFG}$ as the operator associated to a pair Bessel (frame) $ (m,G,F)$. 

\begin{defn}\label{d: pair frame} 
For that $F=\{f_i\} $, $ G=\{g_i\}$ and $m=\{m_i\}$ as in Definition \ref{d: pair Bessel}, let $ (m, G,F )$ be an $m$-pair Bessel for $\h$ . We say that $(m,G,F)$ is an \textbf{$m$-pair frame} for $\h$ if the operator $S_{mFG}$ defined in (\ref{e: pair Bessel}) is invertible. In the case of $m=\{1\}$, $( G,F )$ will be sail to be a \textbf{pair frame}.
\end{defn} 
Similar to the $m$-pair Bessels, with respect to the type of convergence of the series in (\ref{e: pair Bessel}), we can define \textbf{unconditional $m$-pair frame}.
If $(m, G,F )$ is an $m$-pair Bessel (frame) for $\h$ , then $(\overline{m}G,F)$ and $(G,mF)$ are pair frame (Bessel) for $\h$. Also, every pair Bessel (frame) is an $m$-pairpair frame (Bessel) by putting $ m=\{1 \}$. For this, sometimes instead of calling $(m, G,F )$ an $m$-pair pair frame (Bessel), we call it a pair pair frame (Bessel) simply. For this, all results about pair frame (Bessel) are valid for $m$-pair pair frame (Bessel) and vice versa. 


In the case that $(m, G,F )$ is an $m$-pair Bessel, Balazs \cite{2007BasicdefBalazs} called $S_{mFG}$ a \emph{multiplier operator}. Also the invertibility of these operators are studied in \cite{2011StoevaBalazsInvertibilityofmultiplier}.

There are examples for which $S_{mFG}^* \neq S_{\overline{m}GF}$; even well-definedness of $S_{mFG}$ does not imply the well-definedness of $ S_{\overline{m}GF}$ in general \cite[Remark 3.4]{2010UnconditionalConveStoeva}.
The equality $S_{mFG}^* = S_{\overline{m}GF} $ holds under some certain conditions considered in Theorem \ref{t: adjoint frame}. For proving that theorem, we need the following lemma.

\begin{lem} \label{l: uncond convergence series} 
\cite{1998BasisTheoryPrimerHeil,1993UberunbedingteOrlicz,1938OnIntegrationPettis} For a sequence $\{h_i\}\subset\mathcal{H}$, the following statements are equivalent:
\begin{enumerate} 
\item $\sum_i h_i $ converges unconditionally.
\item $\sum_j h_{j} $ converges for every $ \{h_{j}\} \subset \{h_i\}. $ 
\item $\sum_j h_{j} $ converges weakly for every $ \{h_j\} \subset \{h_i\}. $ 
\end{enumerate} 
\end{lem} 

\begin{thm}\label{t: adjoint frame} 
For two sequences $F=\{f_i\}$, $G=\{g_i\} \subset\h$ and a scalar sequence $m=\{m_i\}$,
\begin{enumerate} 
\item If $ (m, G,F )$ and $(\overline{m},F,G)$ are pair Bessels for $\h$, then $S_{mFG}^*=S_{\overline{m}GF}$. Additionally, in this case when $(m, G,F )$ is a pair frame, so is $(\overline{m},F, G )$.
\item $(m,G,F)$ is an unconditional pair Bessel (frame) for $\h$ if and only if $(\overline{m},F,G) $ is an unconditional pair Bessel (frame) for $\h$. In this case $S_{mFG}^*=S_{\overline{m}GF}$.

\end{enumerate} 
\end{thm} 
\begin{proof} 
(1). Since $(m, G,F )$ and $(\overline{m},F,G)$ are pair Bessels, $S_{mFG}$ and $ S_{\overline{m}GF}$ are well-defined and for every $f,g\in \mathcal{H}$,
\begin{align*} 
\langle S_{\overline{m}GF} (f) , g \rangle=\langle \sum_i \langle f, f_i \rangle g_i , g \rangle & = \sum_i \langle \overline{m_i} \langle f, f_i \rangle g_i , g \rangle \\ 
& =\sum_i \langle f , g_i \rangle f_i =\langle f , S_{mFG} (g) \rangle .
\end{align*} 
Thus $S_{mFG}^*=S_{\overline{m}GF}$ . \\ 
For the pair frame case, invertibility of $S_{mFG}$ results invertibility of $S_{mFG}^* =S_{\overline{m}GF}$.\\ 
(2). Theorem \ref{t: adjoint frame} is proved in an informally published article \cite{2010UnconditionalConveStoeva}. For the sake of completeness the proof is stated completely here. $(m,G,F)$ is an unconditional pair Bessel for $\h$ if and only if for every $f\in \mathcal{H}$, $\sum_i m_i \langle f , f_i \rangle g_i $ converges unconditionally for every $f\in \mathcal{H}$. By Lemma \ref{l: uncond convergence series} it is equivalent to the fact that $\sum_{i\in J} m_{i_{}}\langle f , g_i\rangle f_i $ converges for every subset $\mathbb{J}$ of $\mathbb{I}$ and $f\in \mathcal{H}$. Again by Lemma \ref{l: uncond convergence series} this means that 
\[ \langle \sum_{i\in \mathbb{J}} m_{i_{}} \langle f, f_i\rangle g_i , g \rangle =\sum_{i\in \mathbb{J}} \langle f , \overline{m_i} \langle g, f_i\rangle g_i \rangle . \] 
for every subset $\mathbb{J}$ of $\mathbb{I}$ and $G,F\in \mathcal{H}$.
This means for every subset $\mathbb{J}$ of $\mathbb{I}$ and $g\in \mathcal{H}$, $\sum_{i\in J} \overline{m_{i_{}}} \langle g , f_i\rangle g_i$ converges weakly. So the above lemma implies that $S_{\overline{m}GF}$ is unconditionally well-defined. In the other word $(\overline{m},F,G)$ is an unconditional pair Bessel.

The frame case is a consequence of the invertibility of the adjoint of an invertible operator.
\end{proof} 
More studies on this field are done in \cite{2012AdjointofPairFrames-A.FereydooniA.Safapour}, where the concept of adjoint of pair frames is introduced in Banach space setting. If we use only a single sequence $F=\{f_i\}$ in pair frame (Bessel) definition instead two sequence $F=\{f_i\}$, $G=\{g_i\}$ and put $m=\{1\}$, we see that frames (Bessel sequences) are merely pair frames (Bessels). 
In fact:
\begin{prop} \label{} 
A sequence $F=\{f_i\} \in \h $ is a frame (Bessel sequence) for $\h$ if and only if $(F,F)$ is a pair frame (Bessel) for $\h$.
\end{prop} 
\begin{proof} 
It is an straightforward consequence of Theorem \ref{t: frame} and definition of pair frames.
\end{proof} 

\textbf{First}, the above proposition says that pair frames (Bessels) are generalizations of frames (Bessel sequences). But the main importance of pair frames is not this. \textbf{Second}, Pair frames help us to get new reconstruction formulas like as frames and bases. For $F=\{f_i\}, G=\{g_i\}\subset \h$ and suppose that $(G,F)$ is a pair frame for $\h$ with the associated operator $S$. Invertibility of $S$ is the key property of a pair frame $(G,F)$ for obtaining reconstruction formulas: 

\begin{equation}\label{e:reconstruction formulas of pair frames} 
f= \sum \langle f, g_i\rangle S^{-1} f_i, \s f= \sum \langle f, {S^{-1}}^* g_i\rangle f_i. \s 
\end{equation} 
Advanced methods invented for computing inverse of frame operator and frame expansion, can be conveniently applied to compute the above expansions. 

\textbf{Third}, in practise, what we deal with is pair frames not frames. Consider frame $F=\{f_i\}$ and its operator 
$$ Sf= \sum \langle f, f_i\rangle f_i .$$ 
Computing this expansion (operator) accompany some perturbations: 
$$ S'f= \sum \langle f, {f_i '}^{} \rangle f_i'' .$$ 
It is necessary $S'$ be invertible to obtain reconstruction formulas like as (\ref{e:reconstruction formulas of pair frames}). Well-definedness of $\sum$-expansion and invertability of the operator is elements of pair frame definition of $(F',F'')$; where $F'=\{f_i'\}$ and $F''=\{f_i''\}$.


Let $V \in \mathcal{B(H)}$. The operator $V$ is called bounded below if 
$$ 0 <\inf_{} \{ \|V(f) \| \ | \ f \in \h, \ \|f\|=1 \}. $$ 
For $V \in \mathcal{B(H)}$, define 
\[ |\lfloor V \rfloor| = \inf_{} \{ \|V (f) \| \ | \ f \in \h, \ \|f\|=1 \}. \] 
Hence $V$ is bounded below if and only if $0 <|\lfloor V \rfloor|$.
For constants $A,B>0$ , we write 
\[ A \leqslant |\lfloor V \rceil| \leqslant B, \] 
whenever for every $f \in \h $,
\[ A\|f\| \leqslant \parallel V (f) \parallel \leqslant B \|f\| . \] 
\begin{lem}\label{l: invertibility, bounded below} 
Let $V\in \mathcal{B}(\mathrm{\h})$. The following statements are equivalent:
\begin{enumerate} 
\item $V$ is invertible.
\item $V$ and $V^*$ are bounded below.
\item $V$ and $V^*$ are injective and have closed ranges.
\end{enumerate} 
\end{lem} 
\begin{proof} 
$(1) \Rightarrow (2).$ When $V$ is invertible so is $V^*$. Invertible operators are bounded below. \\ 
$(2) \Leftrightarrow (3).$ It is proved in \cite[III.12.Ex5]{1980FuncConw}. \\ 
$(3) \Rightarrow (1)$. It suffices to show that $S$ is onto. Since $V^*$ is one to one and $V$ has closed range 
\[\mathcal{R}(V)=\overline{\mathcal{R}(V)}=\mathcal{N}(V^*)^{\bot}=\mathrm{\h}.\] 
\end{proof} 

We combine Theorem \ref{t: frame}, Lemma \ref{l: invertibility, bounded below} and introduced notations to get some equivalent conditions for a Bessel sequence to be a frame based on its frame operator.

\begin{cor} \label{c: frames} 
Suppose that $F=\{f_i\}$ is a Bessel sequence for $\h$. Then the following statements are equivalent:
\begin{enumerate} 
\item $F=\{f_i\}$ is a frame for $\h$.
\item There exist constants $A,B>0$ such that 
\[ A\leqslant |\lfloor S_F \rceil| \leqslant B. \] 
\item There exist constants $A,B>0$ such that 
\[ A\leqslant S_F \leqslant B.\] 
\item $S_F $ is injective and has closed range.
\item $S_F$ is surjective. 
\item $S_F$ is invertible.

\end{enumerate} 
\end{cor}



\begin{thm}\label{t: below inequality} 
Let $F=\{f_i\}$, $G=\{g_i\} \subset \h$ and $m=\{m_i\}$ be a scalar sequence. Assume that $(m, G,F )$ is a pair Bessel for $\h$.
\begin{enumerate} 
\item If there exists a constant $A > 0$ such that for every $f \in \h $ 
\begin{equation}\label{e: pair frame inequlity 0} 
A \|f\|^2 \leqslant | \sum_{i} m_i\langle f , g_i \rangle \langle f_i , f \rangle | ,
\end{equation} 
then $(m,G,F)$ is a pair frame for $\h$. In this case we obtain a frame-like inequalities i.e. there will be a constant $B>0$ such that for every $f \in \h $ 
\begin{equation}\label{e: pair frame inequlity 1} 
A \|f\|^2 \leqslant | \sum_{i} m_i\langle f , g_i \rangle \langle f_i , f \rangle | \leqslant B \|f\|^2 .
\end{equation} 
\item Additionally suppose that $( \overline{m}, F, G )$ is a pair Bessel for $\h$. The triple $(m, G,F )$ (or $(\overline{m},F,G)$) is a pair frame for $\h$ if and only if there are constants $A,B,A',B' > 0$ such that for every $f \in \h $ 
\begin{equation}\label{e: pair frame inequlity 2} 
A \|f\|^2 \leqslant | \sum_{i} m_i\langle f , g_i \rangle \langle f_i , f \rangle | \leqslant B \|f\|^2 .
\end{equation} 
and 
\begin{equation}\label{e: pair frame inequlity 3} 
A' \|f\|^2 \leqslant | \sum_{i} \overline{m_i} \langle f , f_i \rangle \langle g_i , f \rangle | \leqslant B' \|f\|^2 .
\end{equation} 
\end{enumerate} 

\end{thm} 
\begin{proof} 
Since $(m, G,F )$ is a pair Bessel for $\h$, $S_{mFG}$ is well-defined and for $f \in \mathcal{H}$,
$$ \langle S_{mFG} (f) , f \rangle= \sum_{i} m_i\langle f , f_i \rangle \langle f_i , f \rangle . $$ 
(1). By Cuachy-Schwarz inequality we have 
\begin{equation}\label{e: 21} 
A \|f\|^2 \leqslant |\langle S_{mFG} (f) , f\rangle | \leqslant \| S_{mFG} (f) \| \| f \| .
\end{equation} 
for every $f \in \h $. Also for every $f \in \h $ 
\begin{equation}\label{e: 23} 
A \|f\|^2 \leqslant | \langle f, S_{mFG}^* (f) \rangle | \leqslant \| S_{mFG}^* (f) \| \| f \| .
\end{equation} 
Relations (\ref{e: 21}) and (\ref{e: 23}) imply that $ S_{mFG} $ and $ S_{mFG}^* $ are bounded below. Therefore Lemma \ref{l: invertibility, bounded below} shows that $ S_{mFG} $ is invertible and hence $(m,G,F)$ is a pair frame. Now, put $ B=\| S_{mFG} \| $ to obtain (\ref{e: pair frame inequlity 1}). \\ 
(2). Since $( \overline{m}, F , G ) $ is a pair Bessel for $\h$, $S_{\overline{m}GF} $ is well-defined and for $f \in \mathcal{H}$,
$$ \langle S_{\overline{m} GF} (f) , f \rangle= \sum_{i} \overline{m_i} \langle f , f_i \rangle \langle f_i , f \rangle . $$ 

Assume that $(m,G,F)$ is a pair frame for $\h$. By Theorem \ref{t: adjoint frame} one can conclude that $(\overline{m}, F,G)$ is also a pair frame for $\h$. Hence the operators $ S_{mFG}$ and $ S_{\overline{m} GF}$ are invertible and therefore bounded below by Lemma \ref{l: invertibility, bounded below}. By this and the fact that $ S_{mFG}$ and $ S_{\overline{m} GF}$ are bounded we conclude that there are constants $A,B,A',B' > 0$ such that relations (\ref{e: pair frame inequlity 2}) and (\ref{e: pair frame inequlity 3}) holds.

Conversely, assume that (\ref{e: pair frame inequlity 2}) and (\ref{e: pair frame inequlity 3}) holds for some constants $A,B,A',B' > 0$. Then both of $ S_{mFG}$ and $ S_{\overline{m} GF}$ are bounded below and therefore by Lemma \ref{l: invertibility, bounded below} they are invertible. Hence $(m,G,F)$ and $( \overline{m}, F,G)$ are pair frames for $\h$.
\end{proof} 
The frame inequality (\ref{e: ordinary frame}) can be written in the form 
\begin{equation}\label{e: g-frame second form} 
A\|f\|^2\leqslant | \sum_{i} \langle f, f_i \rangle \langle f, f_i \rangle | \leqslant B\|f\|^2 .
\end{equation} 
It can be seen that the inequalities in Theorem \ref{t: below inequality} are similar to the inequalities in ordinary frame definition (see(\ref{e: g-frame second form})). In fact, if instead of using the \emph{pair}s $F$ and $G$, one put $F=G$ and $m=\{1\}$, then all of the relations (\ref{e: pair frame inequlity 1}),(\ref{e: pair frame inequlity 2}) and (\ref{e: pair frame inequlity 3}) coincide with the inequalities in the definition of frames (see(\ref{e: g-frame second form})).


If $V \in \mathcal{B(H)}$, for the real constants $A,B$, let we write 
\begin{equation}\label{e: abs-frame-ineq} 
A \leqslant |\langle V \rangle | \leqslant B,
\end{equation} 
whenever 
\[ A \|f\|^2 \leqslant | \langle V (f) , f\rangle | \leqslant B \|f\|^2 \quad\quad \forall f \in \mathcal{H}. \] 
When $V= V^*$ and $A,B\geq 0 $,
\[ A \leqslant V \leqslant B, \] 
if and only if 
\[A \leqslant | \langle V \rangle | \leqslant B.\] 
Because 
$$ 0 \leqslant \langle V (f) , f\rangle = | \langle V (f) , f\rangle|, \s \forall f \in \h.$$ 
Namely, when $V= V^*$ and $A,B\geq 0 $, the notations $A \leqslant V \leqslant B$ and $ A \leqslant | \langle V \rangle | \leqslant B$ are the same. For example in the that $ F$ is a frame, put $V=S_{F}$. With this notations, we can summarize the above discussion in the following corollary.
\begin{cor} \label{c: frames2} 
Let $F=\{f_i\}$,$G=\{g_i\} \subset \h $ and $m=\{m_i\}$ be a scalar sequence. Then the following statements are equivalent:
\begin{enumerate} 
\item $(m,G,F)$ and $(\overline{m},F,G)$ are pair frames for $\h$.
\item There exist constants $A,B,A',B' > 0$ such that 
\[ A\leqslant |\lfloor S_{mFG} \rceil| \leqslant B, \s A'\leqslant |\lfloor S_{\overline{m}GF}\rceil| \leqslant B'. \] 
\item There exist constants $A,B,A',B' > 0$ such that 
\[ A\leqslant |\langle S_{mFG} \rangle | \leqslant B, \s A'\leqslant |\langle S_{\overline{m}GF}\rangle | \leqslant B' . \] 
\item $ S_{mFG} $ and $ S_{\overline{m}GF}$ are injective and have closed ranges.
\item $ S_{mFG} $ and $ S_{\overline{m}GF}$ are surjective.
\item $ S_{mFG} $ and $ S_{\overline{m}GF}$ are invertible.

\end{enumerate} 
\end{cor} 
In the other words, corollary \ref{c: frames2} extends the results of corollary \ref{c: frames} from frames to pair frames. In \cite{2012BanachPairFrames-A.FereydooniA.Safapour}, where Banach pair frames is introduced, the results of corollary \ref{c: frames2} are considered for Banach pair frames and  the relation  between concept of pair frames and some   other definitions in the frame theory such as atomic decomposition and  Banach frame are characterized. 
\vskip 0.8 true cm 
\section{\bf {\bf \em{\bf $(p,q)$-Pair Frames}}} 
In this section, we introduce an important class of the pair Bessels and pair frames.
\begin{defn}\label{d: p-fram bassel} 
Let $F=\{f_i\} \subset \h$ and $ 1 \leqslant p <\infty $. The sequence $F$ is called a \textbf{$p$-frame} for $\h$ if there are constants $A,B>0$ such that for every $f \in \h $,
\begin{equation}\label{e: p-frames }
A \|f\|^p \leqslant \sum_i | \langle f , f_i \rangle |^p \leqslant B \|f\|^p .
\end{equation} 
$A$ and $B$ are called upper and lower $p$-frame bounds, respectively. If the right hand inequality of (\ref{e: p-frames }) is satisfied for some constant $B>0$, $F$ is called a \textbf{$p$-Bessel sequence} for $\h$ with bound $B$.

\end{defn} 
$p$-frames in Banach spaces are considered in \cite{2003pframesinseparableBanachspChristensenStoeva,2012 Ehler  Okoudjou Minimization of the probabilistic p-frame potential}. Also, Bachoc, and  Ehler \cite{2013 Bachoc  Ehler  Tight p-fusion frames}   considered tight p-frames.

\begin{defn}\label{d: (m,p,q)-pair Bessel} 
Let $F=\{f_i\}$, $G=\{g_i\} \subset \h $ and $m=\{m_i\}$ be a scalar sequence. If $ 1\leqslant p,q < \infty $ with $ 1/p+1/q=1 $, we say that $(m,G,F)$ is an \textbf{ $m$-$(p,q)$-pair Bessel} for $\h$ if $m \in \ell^\infty$ and $F$,$G$ are $p$-Bessel sequence and $q$-Bessel sequence, respectively. When $m=\{1\}$, $(G,F)$ will be said to be a $(p,q)$-pair Bessel.
\end{defn} 
If $F=\{f_i\}$ is a Bessel sequence, it is a $(2,2)$-pair Bessel sequence. Hence Bessel sequences constitute a subclass of the pair Bessels. In what follows, we show that $m$-$(p,q)$-pair Bessels are really pair Bessels; in fact, they are unconditionally pair Bessels.
\begin{thm}\label{t: (p q)-pair frames }
Let $F=\{f_i\}$, $ G=\{g_i\} \subset \h $ and $m=\{m_i\}$ be as Definition \ref{d: (m,p,q)-pair Bessel}. If $(m,G,F)$ is an $m$-$(p,q)$-pair Bessel for $\h$,
then it is an unconditionally $m$-pair Bessel and 
$$\|S_{m FG}\| \leqslant {\parallel m \parallel}_\infty B^{\frac{1}{p}} B'^{\frac{1}{q}} ,$$ 
where $B$, $B'$ are Bessel sequence bounds of $F=\{f_i\}$ and $ G=\{g_i\} $, respectively.
\end{thm} 
\begin{proof} 
According to the definition, $F $ and $ G $ are $p$-Bessel sequence and $q$-Bessel sequence, respectively. For a finite set $\mathbb{J}\subset \mathbb{I}$ and $f \in \mathcal{H}$ put $ g=\sum_{i \in \mathbb{J} } m_i \langle f , f_i \rangle g_i$. Then 
\begin{align*} 
\|\sum_{i \in \mathbb{J}} m_i \langle f, f_i\rangle g_i \|^2 & = |\langle g,g \rangle | = |\sum_{i \in \mathbb{J}} \langle m_i \langle f, f_i\rangle g_i , g \rangle | \\ 
& \leqslant \sum_{i \in \mathbb{J}} | m_i \langle f , f_i \rangle \langle g_i , g \rangle | \leqslant \sum_{i \in \mathbb{J}}| m_i | | \langle f , f_i \rangle | | \langle g_i , g \rangle | \\ & \leqslant 
(\sup_i |m_i|) (\sum_{i \in \mathbb{J}} | \langle f , f_i \rangle |^q)^{\frac{1}{q}} (\sum_{i \in \mathbb{J}} |\langle g_i , g \rangle |^p )^{\frac{1}{p}} \\ 
& \leqslant (\sup_i |m_i|) B'^{\frac{1}{q}} \|g\|(\sum_{i \in \mathbb{J}} |\langle f,f_i \rangle |^p )^{\frac{1}{p}}.
\end{align*} 
This implies that $\sum_{i \in \mathbb{J}} m_i \langle f, f_i\rangle g_i$ converges unconditionally for every $f \in \mathcal{H}$. Thus $ S_{m FG} $ is well-defined unconditionally. The above computations imply that 
\[ \|S_{m FG } (f) \|= \|\sum_{i } m_i \langle f, f_i\rangle g_i \| \leqslant {\parallel m \parallel}_\infty B^{\frac{1}{p}} B'^{\frac{1}{q}} \|f \|.\] 
\end{proof} 
In the case that $(m,G,F)$ is an $m$-$(p,q)$-pair Bessel, $(m,G,F)$ will be called an \textbf{$m$-$(p,q)$-pair frame} if $ S_{mFG} $ is invertible. In the other contexts $p$-Bessel sequence and $q$-Bessel sequence are called $\ell^p$-Bessel sequence and $\ell^q$-Bessel sequence. Theorem \ref{t: (p q)-pair frames } shows that $\ell^p$-Bessel sequences and $\ell^q$-Bessel sequences are pairable i.e. a pair of sequences of these types can construct a pair Bessel. If we let $\ell=\ell^p$, then $\ell^*=\ell^q$ and $\ell$,$\ell^*$ are pairable. For a more general scalar sequence space $\ell$, it is proved that $\ell$ and $\ell^*$ are also pairable \cite{2012AdjointofPairFrames-A.FereydooniA.Safapour}. \\ 
\section{\bf {\bf \em{\bf Near Identity Pair Frames }}}
Christenson and Laugesen \cite{2009ApproxChris} introduced the notion of " approximately dual frames " in the context of ordinary Bessel sequences. Here we present a new notion similar to " approximately dual frames " which extends that notion. Also some results of \cite{2009ApproxChris} are generalized in this section.
\begin{defn}\label{d: near identity pair frame} 
Let $F=\{f_i\}$,$G=\{g_i\} \subset \h $ and $m=\{m_i\}$ be a scalar sequence. A pair Bessel $(m,G,F)$ will be called a \textbf{near identity pair frame} for $\h$ if for a nonzero $ \alpha \in \mathbb{C} $, 
\begin{equation}\label{e:near identity pair frame} 
\| I - \alpha S_{ m FG } \| < 1 .
\end{equation} 
The triple $(m,G,F)$ will be said to be a \textbf{positively near identity pair frame} for $\h$ if $ \alpha \in (0,\infty) $ and $ S_{ m FG } $ is self adjoint.
\end{defn} 
In Theorem \ref{t: near identity and pair frames and operators}(1), it is proved that every near identity pair frame is really a pair frame. The operator $S_{ m FG }$ corresponding to a near identity pair frame $(m,G,F)$ can be regarded as a perturbation of the identity operator. With our definition, " approximately dual frames " are " positively near identity $(2,2)$-pair frame". In the other words, approximately dual frames are a very special case of pair frames or even near identity pair frames. Also every frame is a positively near identity $(2,2)$-pair frame; see \ref{t: frame}(3). 
The following result shows that the nature of positively near identity pair frames are very similar to ordinary frames.

\begin{prop} \label{c: }
Let $F=\{f_i\}$, $G=\{g_i\} \subset \h $ and $m=\{m_i\}$ be a scalar sequence. $(m,G,F)$ is a positively near identity pair frame for $\h$ if and only if there are constants $A,B>0$ such that 
\[ A \leqslant S_{m FG} \leqslant B. \] 
\end{prop} 
\begin{proof} 
Since $S_{ m FG }$ is selfadjoint and $\alpha$ is positive, a reasoning like the proof of the equivalence $(2) \Leftrightarrow(3)$ of Theorem \ref{t: frame}, establishes the claim. It is enough to replace $S_{ F} $ with $S_{ m FG }$.
\end{proof} 
Figure (1) shows the relations between some different kinds of frames proposed here. 
\begin{center} 
\includegraphics[width=3in]{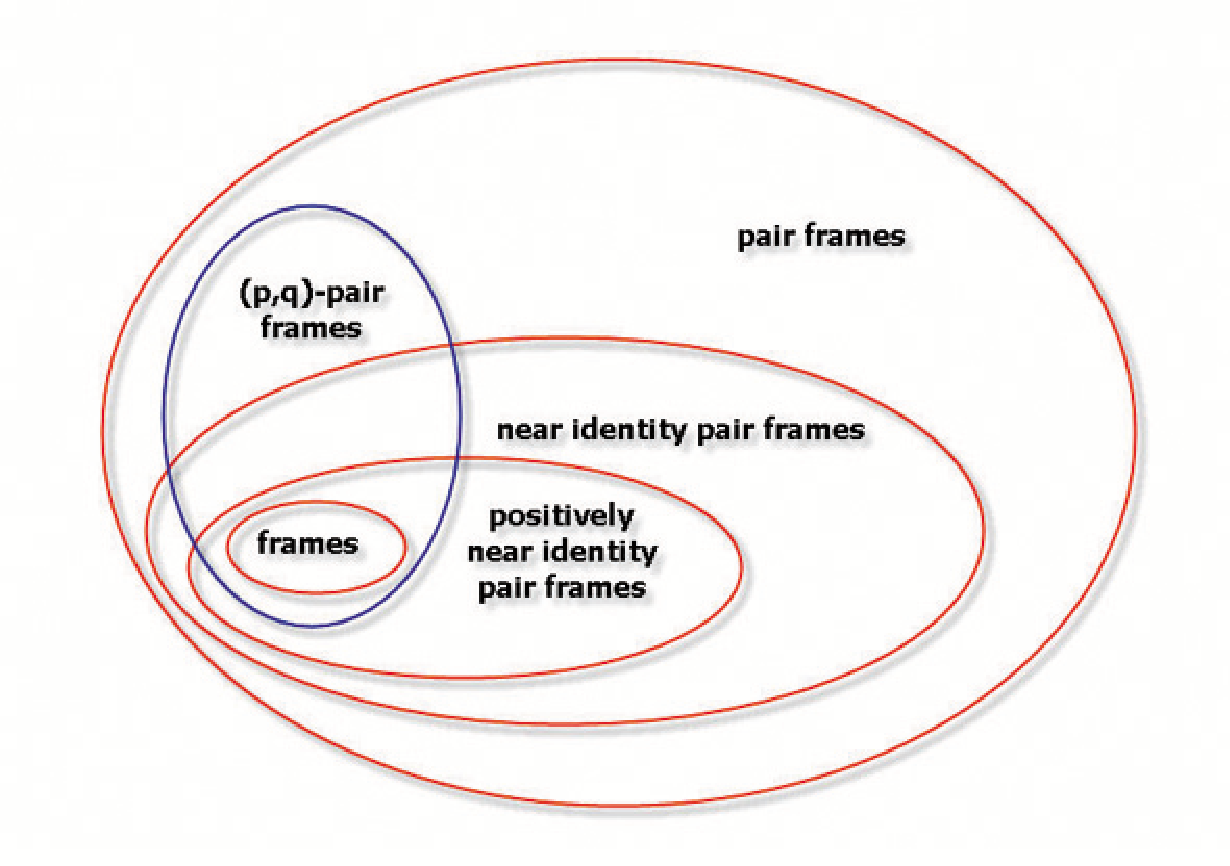} 
\end{center} 
\begin{center} 
Figure 1.
\end{center} 



\begin{thm}\label{t: near identity and pair frames and operators} 
Let $F=\{f_i\}$,$G=\{g_i\} \subset \h$ and $m=\{m_i\}$ be a scalar sequence. Also suppose that $V,W \in \mathcal{B}(\mathcal{H})$. Then 
\begin{enumerate} 
\item If $(m,G,F)$ is a near identity pair frame, then it is a pair frame.
\item If $(m,G,F)$ is a pair Bessel, then $(m, W G, V F)$ is a pair Bessel.
\item If $(m,G,F)$ is a pair frame and V,W are invertible, then $(m, W G, V F)$ is a pair frame.
\end{enumerate} 
\end{thm} 
\begin{proof} 
Recall that when $(m,G,F)$ is a pair Bessel, then the operator $S_{m FG} $ is well-defined. \\ 
(1). If 
\[ \| I - \alpha S_{ m FG } \| < 1 , \] 
for some nonzero $ \alpha \in \mathbb{C} $, then $\alpha S_{m FG} $ is invertible, \cite[VII.2.1]{1980FuncConw}. Since $ \alpha$ is nonzero, then $S_{m FG} $ is invertible and hence $(m,G,F)$ is a pair frame. \\ 
(2). The equation 
\[ V S_{m FG} W^* (f) = \sum_i m_i {V}^* \langle W^* f , g_i \rangle f_i = \sum_i m_i \langle f , W g_i \rangle {V} f_i = S_{m, VF , WG} (f) . \] 
proves (2).\\ 
(3). We know that the invertibility of $V , $ $S_{m FG}$ and $W$ imply the invertibility of $ V S_{m FG} W^* =S_{m, VF , WG} $. This fact and prove (3).
\end{proof} 
 Kutyniok and Okoudjou studied scalable frames \cite{2013 Kutyniok-Okoudjou Scalable Frames},  those frames $F= \{m_i f_i\}$ which can be a tight frame only by multiplying with a sequence of scalars, $F' = \{m_i f_i\}$. If this situation holds, the frame operator is identity and we do not need to compute inverse of frame operator. But, when we deal with a non-scalable frame, using method of Proposition 4.4. help us for obtaining a sequence of near identity pair frame operator converging to the identity operator.  The importance of the " near identity pair frames " is that the inverse of its corresponding operator $ S_{m FG}$, can be written in " Neumann series " as 
\[ S_{ m FG }^{-1}=\alpha \sum_{n=0}^{\infty}(I-\alpha S_{ m FG })^n , \] 
which is very useful in the computational aspects. For this types of frames next theorem gives a sequence of operators ${\{J^{(N)}\}}_{N=0}^{\infty}$ converging to the identity operator.

\begin{prop}\label{p: }
Let $F=\{f_i\}$,$G=\{g_i\} \subset \h $ and $m=\{m_i\}$ be a scalar sequence. Suppose that $(m,G,F)$ is a near identity pair frame for some nonzero $ \alpha \in \mathbb{C} $. For every $ N \in \mathbb{N} $, define 

\[ {(S_{m FG}^{-1})}_N= \alpha \sum_{n=0}^N { (I - { \alpha S_{m FG}} )}^n, \] 
and let 
\[ J^{(N)}: \h \to \h, \quad J^{(N)}(f) = \sum_{i=1}^{\infty} m_i \langle f , g_i\rangle {(S_{m FG}^{-1}})_N f_i. \] 
Then 
\[ \| I - J^{(N)} \| \rightarrow 0 , \s as \s N \rightarrow \infty , \] 
and 
\[ \| I - J^{(N)} \| \leqslant { \| I - \alpha S_{m FG} \| }^{N+1} , \] 
for every $ N \in \mathbb{N} $. There is a sequence $F_N=\{f_{iN}\}_i \subset \h $ such that $(m,G,F_N)$ is a near identity pair frame for every $ N \in \mathbb{N} $ and their operators $S_{m F_NG}= J^{(N)}$ converges to identity operator as $N \rightarrow \infty $. 

\end{prop} 
\begin{proof} 
By the definition of near identity pair frames we have 
\[ \| I - \alpha S_{ m FG } \| < 1 . \] 
Therefore, 
\[ S_{ m FG }^{-1}=\alpha \sum_{n=0}^{\infty}(I-\alpha S_{ m FG })^n , \] 
and its partial sum
\[ {(S_{m FG}^{-1})}_N= \alpha \sum_{n=0}^N { (I - { \alpha S_{m FG}} )}^n, \] 
is convergent for each $ N \in \mathbb{N} $. For $f \in \h $,
$$ J^{(N)}(f) = \sum_{i=1}^{\infty} m_i \langle f , g_i\rangle {(S_{m FG}^{-1}})_N f_i= {(S_{m FG}^{-1})}_N {S_{m FG}^{}}(f) . $$ 
Therefor $J^{(N)} $ is well-defined for every $ N \in \mathbb{N} $.
So 
\begin{align*} 
J^{(N)} & = {(S_{m FG}^{-1})}_N {S_{m FG}^{}}\\ 
& = ( {\alpha } \sum_{n=0}^{N}(I - \alpha S_{m FG})^n ) S_{m FG} \\ 
& = ( \sum_{n=0}^{N}(I - \alpha S_{m FG})^n ){\alpha } S_{m FG} \\ 
& = ( \sum_{n=0}^{N}(I - \alpha S_{m FG})^n )(I-(I - \alpha S_{m FG} ) ) \\ 
& = I - {(I - \alpha S_{m FG})}^{N+1}. \\ 
\end{align*} 
Hence 
\[ \| I - J^{(N)} \| \leqslant \| {(I - \alpha S_{m FG} )}^ {N+1} \| \leqslant { \| I - \alpha S_{m FG} \| }^{N+1} . \] 

Since $\| I - \alpha S_{m FG} \| < 1$ by the assumption, then $\| I - J^{(N)} \| \rightarrow 0 $ as $ N \rightarrow \infty $.
\end{proof} 
\vskip 0.8 true cm

The authors \cite{2012BanachPairFrames-A.FereydooniA.Safapour} have shown that the  notion of pair frames generalizes some various types of frames. Some characterizations of Banach pair frames are presented in \cite{2012CharacterizationsforBanachPairFrames-A.FereydooniA.Safapour}. Adjoint of pair frames is considered by the authors in \cite{2012AdjointofPairFrames-A.FereydooniA.Safapour}. \\\\


\end{document}